\documentclass[12pt]{article}

\usepackage{lineno,hyperref}
\modulolinenumbers[5]

\usepackage{amsmath, amsthm, amssymb}
\usepackage[ansinew]{inputenc}
\usepackage{amsfonts}
\usepackage{amssymb}
\usepackage{enumitem}
\usepackage{amsthm}
\usepackage{dsfont}
\makeatother

\theoremstyle{plain}
\DeclareMathAlphabet{\mathbbmsl}{U}{bbm}{m}{sl}
\def\diag{\mathop{\rm diag}}

\def\sol{\mathop{\rm SOL}}

\def\A{{\mathcal{A}}}
\def\I{{\mathcal I}}

\def\R{\mathbb{R}^{n\times n}}
\def\Rr{\mathbb{R}^{n}}

\def\S{{\sol{(\mathcal{A},\B,{\bf q}})}}
\def\So{{\sol{(\mathcal{A},\B,{\bf 0}})}}
\def\B{{\mathcal{B}}}

\def\sgn {{\rm sgn}}

\def\x{{\bf{x}}}
\def\xm{{\bf{x}}^{m-1}}
\def\ym{{\bf{y}}^{m-1}}
\def\y{{\bf{y}}}

\def\q{{\bf{q}}}
\def\u{{\bf{u}}}
\def\v{{\bf{v}}}
\def\z{{\bf{z}}}
\def\r{\mathbb{R}^n}

\def\A{\mathcal{A}}

\newtheorem{theorem}{Theorem}
\newtheorem{prop}{Proposition}
\newtheorem{corollary}{Corollary}
\newtheorem{example}{Example}

\newtheorem{lemma}{Lemma}
\newtheorem{definition}{Definition}

\begin{document}
	
	\begin{center}\large{\bf The Horizontal Tensor Complementarity Problem}
	\end{center}
	\begin{center}

		\textsc{Punit Kumar Yadav$^{a}$, Sonali Sharma$^{b}$, K. Palpandi$^{c}$}\\
		$^{a,b}$Department of Mathematics, Malaviya National Institute of Technology Jaipur, India.\\
		$^{c}$Department of Mathematics, National Institute of Technology Calicut, India.\\
		E-mail address: punitjrf@gmail.com\\
	\end{center}


\begin{abstract}
	This article explores a new type of nonlinear complementarity problem, namely the horizontal tensor complementarity problem (HTCP), which is a natural extension of the horizontal linear complementarity problem studied in \cite{szn}. We extend the concepts of ${\bf R_0}$, ${\bf R}$, and ${\bf P}$ pairs from a pair of linear transformations given in \cite{wcp} to a pair of tensors. When a given pair of tensors has these properties, we use degree-theoretic tools to discuss the existence and boundedness of solutions to the HTCP. Finally, we study an uniqueness result of the solution of the HTCP
\end{abstract}

\section{Introduction} Given $A,B\in\R$ and $\q\in\Rr$, the horizontal linear complementarity problem HLCP($A,B,\q$) is to find $(\x,\y)\in\Rr\times\Rr$ such that \begin{equation}\label{eq1}
	\begin{aligned}
		\x\wedge \y&={\bf 0},\\
		A\x -B\y&=\q,\\
	\end{aligned}
\end{equation} where $'\wedge'$ is a pointwise minimum map. The horizontal linear complementarity problem (HLCP) is one of the generalizations of the linear complementarity problem (LCP). The significance of this problem is due to a result by Eaves and Lemke, which states that any piecewise linear system can be expressed as the HLCP, as shown in \cite{lem}. Various numerical methods have been provided to solve the HLCP, as seen in \cite{hlcp,homo}. The HLCP has been used for several years to study interior point algorithms for linear and convex programming, as demonstrated in \cite{gul,kuhn,zhn}. Zhang \cite{zhn} studied the convergence of infeasible interior-point methods for the HLCP. Guler \cite{gul} and Zhang \cite{zhn} also indicated that the HLCP formulation is more useful than the LCP formulation for computing a convex quadratic problem.

Boundedness, existence, and uniqueness results for the HLCP were provided by Sznajder \cite{szn}. He established that the HLCP solution is bounded when the given matrix pair is an ${R_0}$ pair. The matrix pair $\{A,B\}$ is an ${R_0}$ pair if HLCP($A,B,{\bf 0}$) has only the trivial solution. Given that the degree of an $R_0$ pair is nonzero, the HLCP's solution exists. The ${P}$ pair of matrices gives the existence and uniqueness of the solution of HLCP. The matrix pair $\{A,B\}$ is said to be a ${P}$ pair if and only if $A^{-1}B$ is a $P$ matrix. Similar types of outcomes for the HLCP's solution structure on a Euclidean Jordan algebra have been provided by Gowda et al. \cite{wcp}.

In Equation (\ref{eq1}), the HLCP contains the standard LCP as a special case where $A=I$. The LCP has been extensively studied by several authors for several decades, and its numerous applications are well known. To explore the structure of solution sets for the LCP, various classes of matrices have been developed by different authors, for instance, refer to \cite{PP0, fvi}. The class of $R_0$, $Q$, and $P$ matrices is essential to the existence and uniqueness of the LCP solution. More information on these types of matrices can be found in a standard reference by Cottle and Pang \cite{LCP}.

Among other generalizations of LCP, the tensor complementarity problem (TCP) is well known. For $\q\in\Rr$, the TCP($\A,\q$) is to find a vector $\x$ such that $$\x\in \Rr, ~\y=\A\xm + \q\in\Rr ~\text{and}~ \x\wedge \y = {\bf 0},$$ where $\A$ is an $m$-th order and $n$ dimensional tensor given as 
\begin{equation*}\hspace{3.3cm}\label{tensor}
	{\bf \A}=(a_{i_{1}i_{2}...i_{m}}), ~{\text{where}}~i_{j}\in [n]~ \text{for all}~ j \in [m] \text{~and~} a_{i_{1}i_{2}...i_{m}}\in \mathbb{R}.
\end{equation*} 
For $m=2$, the TCP transforms into the LCP. The TCP is a specific class of nonlinear complementarity problems \cite{pang}. The outcomes of nonlinear complementarity theory are applicable to TCP theory. The monograph by Harker and Pang \cite{pang} on the nonlinear complementarity problem is a standard reference. Due to its remarkable structure, TCP has numerous practical applications in several mathematical domains, including economics, game theory, and mathematical engineering. The tensor complementarity problem was introduced by Song \cite{Song}. After that, the TCP was examined swiftly, and several conclusions describing the solution structure of the TCP were discovered by several authors. The notions of ${P}, {R_0}, {R}$, and ${Q}$ matrices have been extended to ${P}, {R_0}, {R}$, and ${Q}$ tensors, see \cite{PTN,t1,t3,t2,stt}. Tensor complementarity problems have been increasingly studied in the last decade (for example, see \cite{GUST,PSDT,kp,ks,tspt}). Also, various generalizations of the TCP have been studied by several authors, (for example, see \cite{vtcp,ptcp,hcp}). The vertical tensor complementarity problem (VTCP) \cite{vtcp} is one of many generalizations of the TCP. The vertical LCP and HLCP have been studied simultaneously in LCP theory. Movtivated by this one can ask whether the HLCP can be generalized to the HTCP. Is there a connection between HTCP and TCP?

The crux of our paper is the extension of the HLCP to the HTCP. We first extend the definition of the HLCP to the HTCP. We define the ${\bf R_0}$, ${\bf R}$, and ${\bf P}$ pair for a pair of tensors and then study the boundedness of the solution set of the HTCP. Using degree theory and ${\bf P}$ tensor pair, we demonstrate the existence of a solution to the HTCP by stating and proving several results.

Here is an outline of our paper: In the next section, we discuss terminology and previous findings. The HTCP is defined in Section 3. Additionally, this section addresses the ${\bf R_0}$ tensor pair and the HTCP degree. The ${\bf R}$ tensor pair is introduced in subsection 3.2 along with an existence result. Lastly, the ${\bf P}$ tensor pair is defined and studied, along with the main existence result. Also, we study an uniqueness result of the solution of the HTCP.
\section{Preliminaries}
Throughout the paper, we use the following notations: \begin{itemize}
	\item[\rm(i)] We denote  $\mathbb{R}^n$ as the $n$ dimensional Euclidean space with the real usual inner product.  We say $\x \geq {\bf 0}$ if $\x\in\Rr$ is a component wise  nonnegative vector.
	\item  [\rm(ii)] $x,y,\alpha,..$ denote real scalars and $\x,\y,\q$,... denote vectors.
	\item[\rm(iii)] $\x*\y$ is the Hadamard product and $\x\wedge\y$ is pointwise minimum map i.e., $\x\wedge\y=\text{min}\{\x,\y\}.$
	\item[\rm(iv)]For $\x=(x_1,x_2,...,x_n)^t\in\Rr$, we write $\x^{[m-1]}=(x_1^{m-1},...,x^{m-1}_n).$
	\item [\rm(v)]  The set of all $n\times n$ real matrices will be denoted by $\R$ and  $A,B,..$ are used to denote real matrices.
	\item [\rm(vi)]  The set of
	all $m$-th order $n$ dimensional real tensors is denoted as $T(m, n)$. The set of all even order real  tensors is denoted as $E(m,n).$ We write tensor as math calligraphic letters such as $\A,\B,..,$ and ${\mathcal I}$ denotes the identity tensor. 
	\item[\rm(vii)]${\rm det}(\A)$ denotes the determinant for a  tensor $\A\in T(m,n)$.
	\item [\rm(vii)] We use $[n]$ to denote the set $\{1,2,...,n\}$.
\end{itemize} 
\vspace*{.5cm}
We now recall some definitions and results from the LCP and TCP theory, which will be used frequently in our paper.

\begin{prop}[\cite{wcp}]\label{star}	For all $\x,\y,\u\in\mathbb{R}^n,$ the followings hold.	\begin{itemize}
		\item [\rm(I)] $\lambda(\x\wedge\y)=\lambda\x\wedge\lambda\y$, for 
		$\lambda\geq 0.$ 
		\item[\rm(II)]${\bf u}+\x\wedge \y=({\u}+\x)\wedge ({\bf u}+\y)$.
		\item[\rm(III)]  The followings are equivalent. 
		\begin{itemize}  		
			\item [\rm(i)]  $\x\wedge \y={\bf 0}.$
			\item [\rm(ii)] $\x,\y\geq {\bf 0}$ and $~\x*\y={\bf 0}$.	\item[\rm(iii)] $\x,\y\geq {\bf 0}~\text{and}~\langle \x,\y\rangle={ 0}.$
		\end{itemize}
\end{itemize}	\end{prop}
\begin{definition}[\cite{Song},\cite{stt}]\rm
	Let $\A\in T(m,n)$. We say $\A$ is a/an
	\begin{itemize}
		\item [\rm(i)] $R_0$ tensor if $[\x\wedge \A\xm={\bf 0}\implies \x={\bf 0}].$
		\item [\rm(ii)] $P$ tensor if $[\x*\A\xm\leq {\bf 0}\implies\x={\bf 0}].$
		\item[\rm(iii)] $R$ tensor if it is an ${R_{0}}$ tensor and  $\mathrm{TCP}({\A},{\bf e})$ has a unique  solution, where ${\bf e}=(1,1,...,1)^t$.
	\end{itemize}
\end{definition} 
\begin{definition}[\cite{gev},\cite{ev}]\rm Let $\A,\B\in T(m,n).$ Then $\lambda\in{\mathbb{R}}$ is called
	\begin{itemize}
		\item [\rm(i)]an H-eigenvalue of $\A$ if there exists $\x\in\Rr\setminus\{\bf 0\}$ such that $\A\xm=\lambda\x^{[m-1]}.$ 
		\item [\rm(ii)]a Z-eigenvalue of $\A$ if there exists $\x\in \Rr\setminus\{\bf 0\}$ such that $\A\xm=\lambda\x$ and $ \x^t \x=1.$ 
		\item [\rm(iii)] a $\B$-eigenvalue of $\A$  if there exists $\x\in\Rr\setminus\{\bf 0\}$ such that   $$\A\x^{m-1}-\lambda\B\x^{m-1}={\bf 0}.$$
	\end{itemize}
\end{definition}
\begin{definition}[\cite{tpro}]\label{tp}\rm	For $\A\in T(m,n),m\geq 2$ and $\B\in T(k,n)$, the tensor product ${\mathcal C}=\A\B\in T((m-1)(k-1)+1,n)$ and is defined as $${c}_{i\alpha_1\alpha_2...\alpha_{m-1}}=\sum_{i_2,...,i_m=1}^n a_{ii_{2}...i_{m}}b_{{{i_2}}{\alpha_1}}...b_{{i_m}{\alpha_{m-1}}},$$ where $i\in[n],~\alpha_1,...,\alpha_{m-1}\in[n]^{k-1}$. This product has the following properties: \begin{itemize}
		\item [\rm(i)] $A(\A_1+\A_2)=A\A_1+A\A_2$, where $A\in\R$ and $\A_1,\A_2\in T(m,n).$
		\item [\rm(ii)] $\A {I}={{I} \A}=\A.$ Here ${I}\in \R$ is the identity matrix.
		\item [\rm(iii)] $\A(\B{\mathcal{E}})=(\A\B){\mathcal E},$ where $\A\in T(m,n),m\geq 2$, $\B\in T(k,n)$ and ${\mathcal E}\in T(l,n).$ 
	\end{itemize}
\end{definition}
\begin{theorem}[\cite{tinv}]\label{minv}
	Let $\A\in T(m,n).$ Then $\A$ has an order $2$ left inverse if and only if there exists a nonsingular $n\times n$ matrix $M$ such that $M\A ={\mathcal I}$. Also, $M$ is the unique  order $2$ left inverse of $\A.$
\end{theorem}


\subsection{Degree theory} In our paper, we rely on degree theory results to demonstrate the existence of solutions for the HTCP. We present some properties of degree theory for our discussion. For further details on degree theory, we refer the readers to \cite{ztn, fvi, deg}.

Let ${\rm\Psi}:\bar{\Delta}\rightarrow \Rr$ be a continuous function, where $\Delta\subseteq\Rr$ is an open bounded set and let ${\bf p}\notin {\rm\Psi}(\partial\Delta)$ be a vector. Here, $\partial\Delta$ and $\bar{\Delta}$ denote the boundary and closure of $\Delta$ respectively. The degree of ${\rm\Psi}$ with respect to ${\bf p}$ over $\Delta$ is denoted by $\deg({\rm\Psi},\Delta,{\bf p})$. The equation ${\rm\Psi}(\x)={\bf p}$ has a solution if $\deg({\rm\Psi},\Delta,{\bf p})\neq 0$. The degree is an integer. If ${\rm\Psi}(\x)={\bf p}$ has a unique solution, say $\y$ in $\Delta$, then the degree is the same over any bounded open sets in $\Delta$ containing the unique solution $\y$. The common degree is called the local degree of ${\rm\Psi}$ and is denoted by $\text{deg}({\rm\Psi},\y)$. If ${\rm\Psi}$ is differentiable with nonsingular derivative at a point $\y$, then (\cite{deg}, page 869)
$$\text{deg}({\rm\Psi},{\y}) = \text{sgn }{\rm det}({\rm\Psi}'({\y})).$$
It is notable that if ${\rm\Psi}$ is a continuous function on $\Rr$ such that ${\rm\Psi}(\x)={\bf 0}$ if and only if $\x={\bf 0}$, then
$\text{deg}({\rm\Psi},{\bf 0}) = \deg({\rm\Psi},\Delta,{\bf 0})$ for
any bounded open set $\Delta$ containing ${\bf 0}$ in $\Rr$.
\subsubsection{Properties of the degree} We list out the following properties for later use.
\begin{itemize}
	\item[(D1)] deg($I,\Delta,\cdot)=1$, where $I$ is the identity function.
	\item [(D2)]     {Homotopy invariance}: Let  $F(\x,t):\Rr\times[0,1]\rightarrow \Rr $  be continuous function and $Z:=\{\x:F(\x,t)={0}~\text{for some}~t\in[0,1]\}$ is the zero set of $F(\x,t).$
	If $Z$ is bounded, then for any bounded open set $\Delta$ in $\Rr$ containing  $Z$, we have $$\text{deg}(F(\cdot,0),\Delta,{ \bf 0})=\text{deg}(F(\cdot,1),\Delta,{\bf 0}).$$
	\item[(D3)] Cartesian product property: Let ${\rm\Psi}:\bar\Delta\rightarrow \Rr$ and ${\rm\Phi}:\bar\Delta'\rightarrow \Rr$  be  two continuous functions, where $\Delta$ and $\Delta'$ are open, bounded and nonempty subsets of $\Rr$ with ${\bf 0}\notin {\rm\Psi}({\partial\Delta})$ and ${\bf 0}\notin {\rm\Phi}({\partial\Delta'})$. Then  the degree of the mapping ${\rm\Psi}\times {\rm\Phi}$ is well defined and given by $$ \deg({\rm\Psi}\times {\rm\Phi},\Delta\times\Delta',({\bf 0},{\bf 0}))  =\deg({\rm\Psi},\Delta,{\bf 0})\deg({\rm\Phi},\Delta',{\bf 0}).$$
\end{itemize}
Now we state the Borsuk's theorem  from (\cite{nfa}, page 78) which will be used later.
\begin{theorem}\label{odd}
	Let $D$ be a symmetric bounded open set in $\Rr$ containing zero and ${\rm\Psi}:D\rightarrow\Rr$ be an odd mapping $({\rm\Psi}(-\x)=-{\rm\Psi}(\x)\text{~for all ~}\x\in D)$ such that ${\bf 0}\notin {{\rm\Psi}}(\partial D)$. Then $\deg({\rm\Psi},D,{\bf 0})$ is an odd number. 
\end{theorem}
\noindent{\bf Note}: All the degree theoretic results and concepts are also applicable over any finite dimensional Hilbert space (like $\Rr$ or $\Rr\times\Rr$ etc).
\section{The horizontal tensor complementarity problem} 

Motivated by equation (\ref{eq1}), we first define the horizontal tensor complementarity problem (HTCP).

\begin{definition}\rm
	Given $\A,\B\in T(m,n)$ and $\q\in\r$, the HTCP$(\A,\B,\q)$ is to find a pair of vectors $(\x,\y)\in \r\times \r$ such that
	\begin{equation}
		\begin{aligned}
			\x\wedge \y&={\bf 0},\\
			\A\xm -\B\ym&=\q.\\
		\end{aligned}
	\end{equation}
\end{definition}
Here, $\text{SOL}(\A,\B,\q)$ denotes the solution set of HTCP($\A,\B,\q$) and is defined as the set of all such $(\x,\y)$.\\

It is noted that HTCP is turned into HLCP for $m=2$. Also, when $m$ is even and $\A={\mathcal{I}}$, HTCP becomes TCP. Next, we define the ${\bf R_0}$, ${\bf R}$, ${\bf P}$ tensor pair in the following subsections.
\subsection{${\bf R_0}$ tensor pair} 
This section defines ${\bf R_0}$ tensor pair and HTCP degree. 
The $R_0$ pair of matrices $\{A,B\}$ is given by 	\begin{equation*}
	[\x\wedge \y={\bf 0},
	A\x -B\y={\bf 0}]\implies (\x,\y)=({\bf 0},{\bf 0}).\end{equation*}
The degree of HLCP with the help of the $R_0$ pair has been defined and discussed in \cite{szn}. When $A=I$ and $B$ is an $R_0$ matrix, the degree of HLCP and degree of LCP are  same. In TCP($\B,\q$), if  $\B$ is an $R_0$ tensor then ${\rm\psi}(\x)=\x\wedge\B\xm={\bf 0}\iff\x={\bf 0}$. So the degree of ${\rm\psi}$ is well defined for any bounded open set $\Delta$ in  $\Rr$  containing zero vector. This common degree is called TCP($\B,\q$)-degree and is denoted by $\deg(\B)$ i.e., $\deg(\B)=\deg({\rm\psi},{\bf 0}).$ Now we define the ${\bf R_{0}}$ tensor pair and its HTCP-degree.
\begin{definition}\rm
	Let $\A,\B\in T(m,n)$. We say $\{\A,\B\}$ is an ${\bf R_{0}}$ tensor pair, if 
	\begin{equation*}
		\begin{aligned}
			[\x\wedge \y={\bf 0},~
			\A\xm -\B\ym={\bf 0}]\implies (\x,\y)=({\bf 0},{\bf 0}).\\
		\end{aligned}
	\end{equation*}Equivalently HTCP$(\A,\B,{\bf 0})$ 
	has  only the  trivial  solution. 
\end{definition}
It is noted that if $\{\A,\B\}$ is an ${\bf R_0}$ tensor  pair, then $${\rm\Psi}(\z)={\bf 0}\iff \z={\bf 0},$$ where \begin{equation}\label{eq3}
	{\rm\Psi}(\z):=\begin{bmatrix}
		\x\wedge\y \\ \A\xm-\B\ym
	\end{bmatrix},~\z=(\x,\y)\in\Rr\times\Rr.\end{equation} Therefore $\text{deg}({\rm\Psi},{\bf 0})$ is well defined. We define the HTCP-degree of tensor pair $\{\A,\B\}$ as $$\deg(\A,\B):=\text{deg}({\rm\Psi},{\bf 0}).$$

Now we give a boundedness result of $\S$ with the help of the ${\bf R_{0}}$ tensor pair. This result extends the boundedness results of the solution set of HLCP \cite{szn}.
\begin{theorem}\label{bdd}
	Let $\A,\B\in T(m,n)$. Then, the following statements are equivalent.
	\begin{itemize}
		\item[\rm(i)] $\{\A,\B\}$ is an ${\bf R_{0}}$ tensor pair.
		\item[\rm(ii)] $\S$ is a compact set for any $\q\in \r$.
	\end{itemize}  
\end{theorem}
\begin{proof}      
	(i)$\implies$(ii): The set $\S$ is always closed. Let  $\{\A,\B\}$ be an ${\bf R_{0}}$ tensor pair. If $\S$ is empty, then we are done. Let $\S$ be a nonempty set.  Now we prove the boundedness of $\S.$  On the contrary, assume $\S$ is unbounded for some $\q$. Let  there exist $\z_k:=(\x_k,\y_k)$ such that $\|\z_k\|\rightarrow \infty,$ as $k\rightarrow \infty,$ with 	\begin{equation*}
		\begin{aligned}
			\x_k\wedge \y_k&={\bf 0},\\
			\A\xm_k -\B\ym_k&=\q.\\
		\end{aligned}
	\end{equation*} Dividing by $\|\z_k\|$ and $\|\z_k\|^{m-1}$ respectively, 
	\begin{equation}\label{eqm}
		\begin{aligned}
			\dfrac{\x_k}{\|\z_k\|}\wedge \dfrac{\y_k}{\|\z_k\|}&={\bf 0},\\
			\A\big(\dfrac{\x_k}{\|\z_k\|}\big)^{m-1} -\B\big(\dfrac{\y_k}{\|\z_k\|}\big)^{m-1}&=\dfrac{\bf q}{\|\z_k\|^{m-1}}.\\
		\end{aligned}
	\end{equation} As $k\rightarrow \infty$, let $$\bar{{\bf z}}=(\bar{\x},\bar{\y})\in\Rr\times\Rr,~\text{where}~\bar{\x}=\text{lim}\dfrac{\x_k}{\|\z_k\|},~\bar{\y}=\text{lim}\dfrac{\y_k}{\|\z_k\|}.$$  
	Therefore in equation (\ref{eqm}), as $k\rightarrow \infty$, 
	\begin{equation*}
		\begin{aligned}
			{\bar \x}\wedge {\bar \y}&={\bf 0},\\
			\A\bar \x ^{m-1}-\B\bar \y ^{m-1}&={\bf 0}.\\
		\end{aligned}
	\end{equation*}
	We see that $\|\bar{\x}\|^2+\|\bar{\y}\|^2=1.$
	But $\z=(\x,\y)=({\bf 0},{\bf 0})$ as $\{\A,\B\}$ is an ${\bf R_{0}}$ tensor pair. Thus we have a contradiction as $(\bar{\x},\bar{\y})\neq({\bf 0}, {\bf 0}).$ Hence $\S$ is bounded for all $\q\in\Rr.$
	
	(ii)$\implies$(i): Let $\z=(\x,\y)\neq ({\bf 0},{\bf 0})$ with,
	\begin{equation*}
		\begin{aligned}
			{\x}\wedge { \y}&={\bf 0},\\
			\A\xm-\B\ym&={\bf 0}.\\
		\end{aligned}
	\end{equation*}
	Multiplying by $\mu$ and $\mu^{m-1}$ respectively, 
	\begin{equation*}
			\begin{aligned}
			{\mu\x}\wedge { \mu\y}&={\bf 0},\\
			\A(\mu\x)^{m-1}-\B(\mu\y)^{m-1}&={\bf 0},\\
		\end{aligned}	\end{equation*}
	where $\mu>{ 0}$. 
	For all $\mu>{ 0}$, we see  $\mu\z=(\mu\x,\mu\y)\neq ({\bf 0},{\bf 0})$ and $\mu \z\in\So$. It contradicts the boundedness of $\So.$ Hence we have our claim.
\end{proof}

\begin{theorem}\label{Rrr}
	Let $\A,\B\in T(m,n)$. If
	\begin{itemize}
		\item [\rm(i)] $\{\A,\B\}$ is an ${\bf R}_0$ tensor pair.
		\item [\rm(ii)] $\deg(\A,\B)$ is nonzero.
		
	\end{itemize}
	Then $\S$ is a nonempty compact set for all  $\q\in\Rr$.
\end{theorem}
\begin{proof} Let  $\{\A,\B\}$ be an ${\bf R}_0$ tensor pair. Due to Theorem \ref{bdd}, it is enough to prove $\S$ is nonempty. Let us consider a homotopy function $G$ such that $$G({\z},t) = \left[\begin{array}{c}
		{\x}\wedge{\y}  \\
		{\A}{\x}^{m-1}-{\B}{\y}^{m-1}-t{\bf q}
	\end{array}
	\right],$${where}~${\z}=({\x},{\y}) \in {\Rr}\times {\Rr},t \in [0,1].$
	
	Then $G({\z},0)= \left[\begin{array}{c}
		{\x}\wedge{\y} \\
		{\A}{\x}^{m-1}-{\B}{\y}^{m-1}
	\end{array}
	\right]$ and $G({\z},1)= \left[\begin{array}{c}
		{\x}\wedge{\y}\\
		{\A}{\x}^{m-1}-{\B}{\y}^{m-1}-{\bf q}
	\end{array}
	\right]$.
	As $\{\A,\B\}$ is an ${\bf R_0}$ tensor pair, by the similar process in  theorem \ref{bdd}, $Z=\{{\bf z}:G({\bf z,t})={\bf 0}\text{ for some }t\in[0,1]\}$ is uniformly bounded. For an open bounded set $\Delta$ containing the set $Z$ in $\Rr\times\Rr$, by the homotopy invariance property of the degree, we have 
	\begin{equation}\label{Rpair}
		\deg({\A},{\B})= \deg(G(\cdot,0),\Delta,{\bf 0})=\deg(G(\cdot,1),\Delta,{\bf 0}).
	\end{equation}
	Due to the fact $\deg(\A,\B)$ is nonzero,$~\deg(G(\cdot,1),\Delta,{\bf 0})$ is nonzero. Thus $G(\cdot,1)={\bf 0}$ has solution in $\Delta$ implying $\S$ is nonempty. Thus $\S$ is both nonempty and compact all  $\q\in\Rr$.
\end{proof}

Now we establish a relation between HTCP-degree and TCP-degree.
\begin{prop}\label{trr} Let $\B\in T(m,n).$ If $\B$ is  an ${ R_{ 0}}$ tensor, then the followings hold.\begin{itemize}
		\item [\rm(i)] $\{\mathcal{I},\B\}$ is an ${\bf R_{\bf 0}}$  tensor pair.
		\item [\rm(ii)] If $m$ is even, then $\deg(\I,\B)=\deg(\B).$
	\end{itemize} 	
\end{prop}

\begin{proof}      
	(i): Let $(\x,\y)\in\mathbb{R}^n \times \mathbb{R}^n$ such that 
	\begin{equation*}		
		\begin{aligned}
			\x\wedge \y&={\bf 0}, \\ \mathcal{I} \x^{m-1}-\B\y^{m-1}&={\bf 0}.\\
		\end{aligned}
	\end{equation*}
	This gives that $\x^{[m-1]}=\B\y^{m-1}.$ As $\x\wedge \y={\bf 0}$, we have  $\x^{[m-1]}\wedge \y={\bf 0}$ which implies that $\B\y^{m-1}\wedge \y={\bf 0}. $ Since $\B$  is an ${ R_{\bf 0}}$ tensor, $\y={\bf 0}.$ So $(\x,\y)=({\bf 0},{\bf 0}).$ Therefore $\{\mathcal{I},\B\}$ is an ${\bf R_{\bf 0}}$ tensor pair.
	
	(ii):  To prove item (ii), we consider the following maps:
	\begin{equation*}
		\begin{aligned}
			F(\z,t)&=\begin{bmatrix}
				\y\wedge(t\x^{[m-1]} +(1-t)\B\y^{m-1}) \\ \mathcal{I} \x^{m-1}-t\B\y^{m-1}
			\end{bmatrix},\\
			H(\z,t)&=\begin{bmatrix}t(\y\wedge\x)+(1-t)(\y\wedge\x^{[m-1]}) 
				\\ \mathcal{I} \x^{m-1}-\B\y^{m-1}
			\end{bmatrix},	\end{aligned}	
	\end{equation*}{where}~$\z=(\x,\y)\in\Rr\times\Rr~{\text{and}}~t\in[{0},1]$. We claim that the zeros of functions $F$ and $H$ are contained in some open bounded set $\Delta\subseteq\Rr\times\Rr$ when $t$ varies from $0$ to $1$. 
	
	To show this,	we proceed further. Regarding to the function $F$, for $t=0$ and $t=1,$ we have $$F(\z,0)=\begin{bmatrix}	\y\wedge\B\y^{m-1}\\ \mathcal{I} \x^{m-1}
	\end{bmatrix} \text{~and~} F(\z,1)=\begin{bmatrix}
		\y\wedge\x^{[m-1]} \\ \mathcal{I} \x^{m-1}-\B\y^{m-1}
	\end{bmatrix}.$$
	As $\B$ is an ${R}_0$ tensor, it is easy to see that $F(\z,t)={\bf 0}\iff \z={\bf 0}$  for any $t\in[0,1]$. For an arbitrary open bounded set $\Delta$ containing zero  in $\Rr\times\Rr$, by the homotopy invariance of degree, \begin{equation}\label{pr0}
		\deg\big(F(\cdot,{0}),\Delta,{\bf 0})=\deg(F(\cdot,1),\Delta,{\bf 0}).	\end{equation}
	As $m$ is even, we see $\y\wedge\x$ and $\y\wedge\x^{[m-1]}$   always has same sign. Thus for $t\in[0,1]$,  \begin{equation}\label{pr1}
		t(\y\wedge\x)+(1-t)(\y\wedge\x^{[m-1]})={\bf 0}\iff \x\wedge\y={\bf 0}.\end{equation}
	Due to the ${\bf R}_0$ pair property of $\{\mathcal{I},\B\}$ and from equation (\ref{pr1}), we get \begin{equation}\label{eq7}
		\{{\bf z}:H(\z,t)={\bf 0},\text{~for some~}t\in[0,1]\}=\{({\bf 0},{\bf 0})\}.
	\end{equation}
	As $({\bf 0},{\bf 0})\in\Delta$, thus by the homotopy invariance of degree and equation (\ref{eq7}), 
	\begin{equation}\label{eq8}
		\deg(H(\cdot,0),\Delta,{\bf 0})=\deg(H(\cdot,1),\Delta,{\bf 0}).\end{equation}
	As $F(\cdot,1)=H(\cdot,0),$ from equations (\ref{pr0}) and (\ref{eq8}), we get,
	$$\deg\big(F(\cdot,{0}),\Delta,{\bf 0})=\deg(F(\cdot,1),\Delta,{\bf 0})=\deg(H(\cdot,0),\Delta,{\bf 0})=\deg(H(\cdot,1),\Delta,{\bf 0}).$$ And $$\deg(\I,\B)=\deg(H(\cdot,1),\Delta,{\bf 0})=\deg(F(\cdot,0),\Delta,{\bf 0}).$$
	Now for two arbitrary open bounded sets $\Delta_1$ and $\Delta_2$ each containing ${\bf 0}$ in $\Rr$ letting $\Delta=\Delta_1\times\Delta_2$, by the Cartesian product (D3), we have $$\deg(\I,\B)=\deg(F(\cdot,0),\Delta,{\bf 0})=\deg(\B)\deg\big({\rm\Phi},\Delta_2,{\bf 0}).$$
	Here ${\rm\Phi}(\x)=\x^{[m-1]}.$ As ${\rm\Phi}$ is a single variable odd function, from theorem 5.1 in \cite{ztn},  $\deg\big({\rm\Phi},\Delta_2,{\bf 0}\big)$  will be one. Therefore, $$\deg(\I,\B)=\deg(\B).$$
	This completes the proof.
\end{proof}

Next, we have  direct consequences of the above proposition.
\begin{corollary}
	Let $\B$ be an even order ${ R}_0$ tensor. If $\deg(\B)\neq 0$, then $\deg(\mathcal{I},\B)\neq 0.$
\end{corollary}
\begin{corollary}
	Let $\B$ be an even order ${ R}_0$ tensor. If $\deg(\B)\neq 0$, then ${\rm HTCP}(\mathcal{I},\B,\q)$ has a nonempty compact solution set for every  $\q\in\Rr$.
\end{corollary}
\subsection{{\bf R} tensor pair}
In this subsection, we first extend the concept of $R$ pair to tensors, and then we study the existence of the solution to the HTCP. To proceed further, we recall the concept of partial symmetrization for a tensor from \cite{weak}.

A tensor ${\A}=(a_{i_{1}i_{2}...i_{m}}) \in T(m,n)$ can always be converted into a tensor $\bar{\A}= (\bar{a}_{{i_{1}i_{2}...i_{m}}})\in T(m,n) $ with 
	$$\bar{a}_{i_1i_2\cdots i_m} = \frac{1}{(m-1)!}\sum\limits_{\tau(i_2\cdots i_m)}{a}_{i_1\tau(i_2\cdots i_m)},$$ where $\tau$ is a permutation on the set $\{i_{2},...,i_{m}\}$ and
\begin{equation}\label{e1}{\A}{\x}^{m-1} = {\bar{\A}{\x}^{m-1}},~\text{for all}~{\x}\in {\Rr}.\end{equation} 
$\bar\A$ is called partial symmetrize tensor. It is  noted that, 
$\nabla({\A}{\x}^{m-1}) = \nabla(\bar{{\A}}{\x}^{m-1}) = (m-1){\bar\A}{\x}^{m-2}$, where $\nabla({\A}{\x}^{m-1})$ is Jacobain matrix of ${\rm\Phi}(\x)=\A\xm$ and   $N={\bar\A}{\x}^{m-2}\in\R$ \cite{qil} with 
$$({\bar\A}{\x}^{m-2})_{ij}= \displaystyle{\sum_{i_{3},...,i_{m}=1}^n}\bar a_{ij i_{3}...i_{m}}x_{i_{3}}...x_{i_{m}}.$$
Also  \begin{equation}\label{e2}
	N\x=({{\bar\A}}{\x}^{m-2}){\x} ={{\bar\A}}{\x}^{m-1}= {{\A}}{\x}^{m-1}.\end{equation} We write $\nabla({\A}{\x}^{m-1})$ at the point ${\x}={\bar{\x}}$ as $\nabla({\A}{\bar{\x}^{m-1}})$. 
\begin{lemma}\label{HTCP}
	Let ${\A},{\B} \in T(m,n)$ and ${\bf q} \in {\Rr}$. If $(\bar{\x},\bar{\y})$ is a solution of the $\mathrm{HTCP}({\A},{\B},{\bf q})$ then $(\bar{\x},\bar{\y})$ is a solution of the $\mathrm{HLCP}(\nabla({\A}{\bar\x}^{m-1}),\nabla({\B}{\bar\y}^{m-1}),(m-1){\bf q})$.
\end{lemma}
\begin{proof}      
	Since $(\bar{\x},\bar{\y})$ is a solution of the $\mathrm{HTCP}({\A},{\B},{\bf q})$, we have $\bar{\x}\wedge\bar{\y}={\bf 0}$ and ${\A}{\bar\x}^{m-1}-{\B}{\bar\y}^{m-1}={\bf q}$. For $(\x,\y)\in\Rr\times\Rr$, due to \ref{e1} and \ref{e2}
	\begin{equation*}
		\begin{aligned}
			(m-1)({\A}{\x}^{m-1}-{\B}{\y}^{m-1}) &= (m-1)({\bar \A}{\x}^{m-1} - {\bar \B}{\y}^{m-1}),\\
			&=(m-1)(({\bar \A}{\x}^{m-2})\x - ({\bar \B}{\y}^{m-2})\y),\\
			&= (\nabla ({\bar\A}{\x}^{m-1})){\x}-(\nabla({\bar\B}{\y}^{m-1})){\y},\\
			&= (\nabla ({\A}{\x}^{m-1})){\x}-(\nabla({\B}{\y}^{m-1})){\y}.
		\end{aligned}
	\end{equation*}
	This gives,  $(\nabla ({\A}{\bar\x}^{m-1})){\bar\x}-(\nabla({\B}{\bar\y}^{m-1})){\bar\y} = (m-1){\bf q}$. Hence $(\bar{\x},\bar{\y})$ is a solution of the ${\mathrm {HLCP}}(\nabla({\A}{\bar\x}^{m-1}),\nabla({\B}{\bar\y}^{m-1}),(m-1){\bf q})$. 
\end{proof}

\begin{definition} Let ${\A},{\B} \in T(m,n)$. We say  $\{{\A},{\B}\}$ is an ${\bf R}$ tensor pair if 
	\begin{enumerate}
		\item[\rm(i)] $\{{\A},{\B}\}$ is an ${\bf R_{0}}$ tensor pair.
		\item[\rm(ii)] There exists ${\bf q} \in {\Rr}$ such that
		\begin{enumerate}
			\item[\rm(a)] $\mathrm{HTCP}({\A},{\B},{\bf q})$ has a unique solution, say, $({\bar{\x}},{\bar{\y}})$,
			\item[\rm(b)] ${\bar{\x}}+{\bar{\y}} > {\bf 0}$, and 
			\item[\rm(c)] $\mathrm{HLCP}(\nabla({\A}{\bar{\x}}^{m-1}),\nabla({\B}{\bar{\y}}^{m-1}),(m-1){\bf q})$ has a unique solution.
		\end{enumerate}
	\end{enumerate}
\end{definition}
Note that, if $\{{\A},{\B}\}$ is an ${\bf R}$ tensor pair, then $\bar{\x} \land \bar{\y} = {\bf 0}$ and ${\bar{\x}}+\bar{\y} > {\bf 0}$. Therefore, without loss of generality, we can say that $\bar{\x}$ and $\bar{\y}$ takes the form ${\bar{\x}} = (\bar{x}_{1}, \bar{x}_{2},...,\bar{x}_{k},0,0,...,0)^t$ and  ${\bar{\y}} = (0, 0,...,0,\bar{y}_{k+1},\bar{y}_{k+2},...,\bar{y}_{n})^t$, for some $k \leq n$. Let us write $\varphi({\x},{\y}) = {\x} \land {\y}$ 
and the partial derivative of $\varphi({\x},{\y})$ with respect to ${\x}$ and ${\y}$, at the point ${\bar{\z}}=({\bar{\x}},{\bar{\y}})$, are denoted by ${\varphi}_{\x}^{'}({\bar{\z}})$ and ${\varphi}_{\y}^{'}({\bar{\z}})$, respectively. We record the following fact about $\varphi({\x},{\y})$  from (\cite{wcp}, page 10)  which will be used in further discussions. If for any $({\bf u},{\bf v}) \in {\Rr}\times {\Rr}$
\begin{equation}\label{Phi}
	{\varphi}_{\x}^{'}({\bar{\z}}){\u}+{\varphi}_{\y}^{'}({\bar{\z}}){\bf v}={\bf 0},
\end{equation} 
then $u_{i} = 0$ whenever $\bar x_{i}=0$ and $v_{i}=0$ whenever $\bar y_{i}=0$. Therefore ${\bf u}$ and ${\bf v}$ takes the form ${\bf u}=(u_{1}, u_{2},...,u_{k},0,0,...,0)^t$ and ${\bf v} = (0,0,...,0,v_{k+1},v_{k+2},...,v_{n})^t.$

We now give an example of an ${\bf R} $ tensor pair.
\begin{example}\rm Let $m$ be even and ${\B} \in T(m,n)$ be an $R$ tensor. We show that $\{{\mathcal{I}},{\B}\}$ is an ${\bf R}$ tensor pair. It is easy to verify that a vector ${\y} \in {\Rr}$ is a solution of the $\mathrm{TCP} ({\B}, {\q})$ if and only if $(({\B}{\y}^{m-1}+{\q})^{[\frac{1}{m-1}]}, {\y})$ is a solution of $\mathrm{HTCP}({\mathcal I},{\B},{\q})$. From definition of $R$ tensor, $\mathrm{HTCP}(\mathcal{I},{\B},{\bf 0})$ has a unique solution $({\bf 0},{\bf 0})$ and the $\mathrm{HTCP}(\mathcal{I},{\B},{\bf e})$ has a unique solution, $({\bf e}^{[\frac{1}{m-1}]},{\bf 0})$. Therefore we get $\{{\mathcal{I}},{\B}\}$ is an ${\bf R}_{0}$ pair and $({\bar{\x}},{\bar{\y}}) =({\bf e}^{[\frac{1}{m-1}]},{\bf 0})$ is the unique solution of the   $\mathrm{HTCP}({\mathcal{I}},{\B},{\bf e})$. Also ${\bar{\x}}+{\bar{\y}} > {\bf 0}$. 
	From Lemma \ref{HTCP}, we get  $({\bar{\x}},{\bar{\y}}) =({\bf e}^{[\frac{1}{m-1}]},{\bf 0})$ is the only solution of $\mathrm{HLCP}(\nabla({\mathcal{I}}{\bar{\x}^{m-1}}),\nabla({\B}{\bar{\y}}^{m-1}),(m-1){\bf e})$. Hence $\{{\mathcal{I}},{\B}\}$ is an ${\bf R}$ tensor pair.
	
\end{example}
\begin{theorem}\label{rdeg}
	Let ${\A},{\B}\in T(m,n)$. If $\{{\A},{\B}\}$ is an ${\bf R}$ tensor pair then, $\S$ is a nonempty compact set for all $\q\in\Rr$.
\end{theorem}
\begin{proof}  Let $\{{\A},{\B}\}$ be an ${\bf R}$ tensor pair. This implies $\{{\A},{\B}\}$ is an  ${\bf R_{0}}$ tensor pair. In view of Theorem \ref{Rrr}, it is enough to prove 	$\deg({\A},{\B})$ is nonzero.
	Let ${\z}=({\x},{\y}) \in {\Rr}\times {\Rr}$, $t \in [0,1]$ and  we consider the function $$G({\z},t) = \left[\begin{array}{c}
		\varphi({\x},{\y})  \\
		{\A}{\x}^{m-1}-{\B}{\y}^{m-1}-t{\bf q}
	\end{array}
	\right],~\text{where}~ \varphi({\x},{\y}) = {\x} \land {\y}.$$
	Similar to theorem \ref{Rrr}, we get the the zero sets of the homotopy $G({\z},t)$ are uniformly bounded as $t$ varies from 0 to 1 and hence contained in some  bounded open subset $\Delta$ of ${\Rr}\times{\Rr}$. Therefore by the homotopy invariance property of the degree, we have 
	\begin{equation}\label{Rpair}
		\deg({\A},{\B})= \deg(G(\cdot,0),\Delta,{\bf 0})=\deg(G(\cdot,1),\Delta,{\bf 0}).
	\end{equation}
	We claim that $\deg(G(\cdot,1),\Delta,{\bf 0}) \neq 0$. Since $\{{\A},{\B}\}$ is an ${\bf R}$ tensor pair, therefore $G({\z},1) = {\bf 0}$ if and only if ${\z}={\bar{\z}}$, where ${\bar{\z}} = ({\bar{\x}},{\bar{\y}})$ is the unique solution of the $\mathrm{HTCP}({\A},{\B},{\bf q})$. Therefore $$\deg(G(\cdot,1),{\bf 0})=\deg(G(\cdot,1),\Delta,{\bf 0})=\deg(\A,\B).$$ We now prove that ${\rm det} G^{'}({\bar{\z}},1)\neq 0$. By taking derivative of $G$ at $\bar{\z}$, we get $$G^{'}({\bar{\z}},1)= \left[\begin{array}{cc}
		{\varphi}_{\x}^{'}({\bar{\z}}) &  {\varphi}_{\y}^{'}({\bar{\z}})\\
		\nabla({\A}{\bar{\x}}^{m-1}) & ~-\nabla({\B}{\bar{\y}}^{m-1})
	\end{array}
	\right].$$ Suppose that $G^{'}({\bar{\z}},1)$ is singular, so there exists nonzero $\tilde{\z}=({\bf u},{\bf v})$ in ${\Rr} \times {\Rr}$ such that $(G^{'}({\bar{\z}},1))\tilde{\z}={\bf 0}$. This implies
	\begin{equation}\label{Phi}
		{\varphi}_{\x}^{'}({\bar{\z}}){\u}+{\varphi}_{\y}^{'}({\bar{\z}}){\bf v}={\bf 0},~\text{and}
	\end{equation}
	\begin{equation}\label{nabla}
		\nabla({\A}{\bar{\x}}^{m-1}){\u} - \nabla({\B}{\bar{\y}}^{m-1}){\bf v} = {\bf 0}.
	\end{equation}
	From equation (\ref{Phi}), we get ${\bf u}=(u_{1}, u_{2},...,u_{k},0,0,...,0)$
	and ${\bf v} = (0,0,...,0,v_{k+1},
	v_{k+2},...,v_{n})$, where $k \leq n$. Therefore  for all sufficiently small $\delta > 0$, we have $({\bar{\x}}+\delta{\u}) \land ({\bar{\y}}+\delta{\bf v}) = {\bf 0}$. Due to $\{{\A},{\B}\}$ being an ${\bf R}$ tensor pair, we have (x bar, y bar ) is the unique solution of HTCP$(\A,\B,\q)$ and from lemma \ref{HTCP}, we see that the $\mathrm{HLCP}(\nabla({\A}{\bar{\x}}^{m-1}),\nabla({\B}{\bar{\y}}^{m-1}),(m-1){\bf q})$ has a unique solution, which is $(\bar{\x},\bar{\y})$. However from equation (\ref{nabla}), we get that $({\bar{\x}}+\delta{\u},{\bar{\y}}+\delta{\bf v})$ is a solution of the $\mathrm{HLCP}(\nabla({\A}{\bar{\x}}^{m-1}),\nabla({\B}{\bar{\y}}^{m-1}),(m-1){\bf q})$, for all sufficiently small $\delta >0$. Hence we get a contradiction to $\{{\A},{\B}\}$ being an ${\bf R}$ tensor pair. Therefore the matrix  $G^{'}({\bar{\z}},1)$ is nonsingular and hence $\sgn ~{\rm det} G^{'}({\bar{\z}},1) \neq 0$. This implies
	$$\deg({\A},{\B})= \sgn ~{\rm det} G^{'}({\bar{\z}},1) \neq 0.$$ Hence $\S$ is a nonempty compact set for all $\q\in\Rr$.
\end{proof}

\subsection{{\bf P} tensor pair}
For ${A,B}\in\R$, we say pair $\{A,B\}$ is  a $P$ pair \cite{PP0} if $$[\x*\y\leq {\bf 0},~A\x-B\y={\bf 0}]\implies (\x,\y)={({\bf0},{\bf 0})}.$$ From \cite{PP0}, this is equivalent to $A^{-1}B$ being a $P$ matrix. The $P$ pair gives the existence and uniqueness of the solution to the corresponding HLCP.

In this section, we define the ${\bf P} ~tensor ~pair$ which extends the notion of $P$ pair of matrix. 
\begin{definition}\rm
	Let $\A,\B\in T(m,n)$. We say   $\{\A,\B\}$ is a ${\bf P}$ tensor pair if $$[ \x*\y\leq {\bf 0},~	\A\xm -\B\ym={\bf 0}]\implies (\x,\y)=({\bf 0},{\bf 0}).$$ 
\end{definition}
\begin{prop}\label{MB}
	Let $M\in\R$ be an order $2$ left inverse of a tensor $\A\in T(m,n)$. Then $\{\A,\B\}$ is an even order ${\bf P}$ tensor pair if and only if $M\B$ is a $P$ tensor.
\end{prop}
\begin{proof} First assume $\{\A,\B\}$ is an even order ${\bf P}$ tensor pair. To prove $M\B$ is a $P$ tensor, let $\y\in\Rr$ with $\y*(M\B)\ym\leq {\bf 0}.$ As $m$ is even, $\y^{[m-1]}*(M\B)\ym\leq {\bf 0}.$ For some $\x\in\Rr$, we write $\x^{[m-1]}=\mathcal{I}\xm=(M\B)\ym.$ As $M$ is the order $2$ left inverse of $\A$ then, by theorem \ref{minv}, $M\A={\mathcal{I}}$. Therefore $(M\A)\x^{m-1}=(M\B)\ym.$ From definition \ref{tp},  $\A\xm -\B\ym={\bf 0}$ and $\x*\y\leq {\bf 0}.$ Thus $(\x,\y)=({\bf 0},{\bf 0})$ as $\{\A,\B\}$ is a ${\bf P}$ tensor pair.  Hence  $M\B$ is a $P$ tensor.\\ For converse,  let $(\x,\y)\in\Rr\times\Rr$ such that $\x*\y\leq {\bf 0},~\A\xm-\B\ym={\bf 0}.$ As $M$ is the order 2 left inverse of $\A$, ${\mathcal I}\xm-(M\B)\ym={\bf 0}$. This implies $\x^{[m-1]}=(M\B)\ym.$ Also $\x*\y\leq {\bf 0}\implies \x^{[m-1]}*\y\leq {\bf 0}.$ From this, $(M\B)\ym*\y\leq{\bf 0}.$ Thus we get $\y={\bf0}$ due to the fact that  $M\B$ is a $P$ tensor. This gives us $(\x,\y)=({\bf 0},{\bf 0}).$ Hence $\{\A,\B\}$ is a ${\bf P}$ tensor pair.
\end{proof}

\begin{theorem}
	Suppose $\A,\B\in E(m,n)$. If $\{\A,\B\}$ is a ${\bf P}$ tensor pair then, the following statements hold.
	\begin{itemize}
		\item [\rm(i)] $\A$ and $\B$ have no zero H-eigenvalues.
		\item[\rm(ii)] All the Z-eigenvalues of  $\A$ and $\B$ are nonzero.
		\item[\rm(iii)] $\A$ has no negative $\B$-eigenvalues.
		\item[\rm(iv)] $\{P\A P^t,~P\B P^t\}$ is a ${\bf P}$ tensor pair for any permutation matrix  $P\in\R$.
	\end{itemize}
\end{theorem}
\begin{proof}(i): If  H-eigenvalue does not exist, it is trivial. Let $\lambda\in\mathbb{R}$ be an H-eigenvalue, and $\x$ be the corresponding eigenvector of  $\lambda$ such that $$\A\xm=\lambda \x^{[m-1]}.$$ Assume $\lambda=0$. As  $\{\A,\B\}$ is an even order ${\bf P}$ tensor pair, $$[\x*{\bf 0}={\bf 0},~ \A\xm-\B{\bf 0}={\bf 0}]\implies (\x,{\bf 0})=({\bf 0},{\bf 0}).$$ It contradicts our assumption. Therefore, $\A$ has no zero H-eigenvalues. Similarly, $\B$ has no zero H-eigenvalues. \\
	(ii):  It follows as similar to item (i).\\
	(iii): In case of the non-existence of $\B$-eigenvalues of $\A$, it is trivial. Otherwise, assume the contrary: there exists $(\lambda, \x) \in \mathbb{R} \times \Rr$ such that $\x \neq {\bf 0}$ and $\lambda < 0$ and $$\A\x^{m-1}-\lambda\B\x^{m-1}={\bf 0},$$ we write it as \begin{equation} \begin{aligned}
			\x*(\lambda\x)&\leq {\bf 0},\\
			\A\x^{m-1}-\B(\lambda'\x)^{m-1}&={\bf 0},\\
		\end{aligned}
	\end{equation}where $\lambda'=\lambda^{\frac{1}{m-1}}$. As $\{\A,\B\}$ is   a ${\bf P}$ tensor pair, then $(\x,\lambda \x)=({\bf 0},{\bf 0}).$ This implies $\x={\bf 0}$, which gives us a contrary argument. Hence $\lambda\geq { 0}.$\\  
	(iv): Let $\z=(\x,\y)\in\Rr\times\Rr $ with  \begin{equation}
		\begin{aligned}
			{\x}* { \y}&\leq{\bf 0},\\
			P\A P^t\xm-P\B P^t\ym&={\bf 0}.\\
		\end{aligned}
	\end{equation}
	As $P$ is a permutation matrix, by the  tensor product in definition \ref{tp}, $P\A P^t\xm=P(\A\x^{m-1}_p)$, $P\B P^t\ym=P(\B\y^{m-1}_p)$, where $P^t\z=(P^t\x,P^t\y)=(\x_p,\y_p)$. Thus we write, \begin{equation}
		\begin{aligned}
			{\x_p}* {\y_p}&\leq{\bf 0},\\
			\A\x^{m-1}_p-\B\y^{m-1}_p&={\bf 0}.\\
		\end{aligned}
	\end{equation}
	As $\{\A,\B\}$ is a ${\bf P}$ tensor pair, $P^t\z={\bf 0}\implies \z=(\x,\y)=({\bf 0},{\bf 0}).$  Hence we have our claim.
\end{proof}

To prove next result we recall the following result (\cite{det}, Theorem 3.1).
\begin{theorem}
Let $\A\in T(m,n)$. Then ${\rm det}(\A)=0$ if and only if $\A\xm={\bf 0}$ for some nonzero $\x\in\Rr.$
\end{theorem}
We present a characterization for an even order ${\bf P}$ tensor pair, which is easily checkable and extends a similar result for $P$ matrix pairs given in \cite{PP0}.
\begin{theorem}
Let $\A,\B\in E(m,n).$ Then, the following statements are equivalent. \begin{itemize}
	\item [\rm(i)] $\{\A,\B\}$ is a ${\bf P}$ tensor pair.
	\item [\rm(ii)] For arbitrary  nonnegative diagonal matrices $D_1,D_2\in\R$ with diag($D_1+D_2)>{\bf 0}$, $${\rm det}(\A D_1+\B D_2)\neq {0}.$$
\end{itemize} 
\end{theorem}
\begin{proof}      (i)$\implies$(ii): Let $\{\A,\B\}$ be a ${\bf P}$ tensor pair. To prove item (ii), assume contrary that there exist $D_1,D_2\in\R$    nonnegative diagonal matrices with  diag($D_1+D_2)>{\bf 0}$ such that, 
$${\rm det}(\A D_1+\B D_2)= {0}.$$  Therefore, for some nonzero $\x\in\Rr$, we have $$(\A D_1+\B D_2)\x^{m-1}= {\bf 0}.$$ Let $d^{(i)}_j$ be the diagonal entry of $D_i$ for $i\in\{1,2\}$ and $j\in[n].$ Now From definition \ref{tp}, for $\x\in\Rr,$ \begin{equation*}
	\begin{aligned}
		(\A D_1\x^{m-1})_i=&\sum_{i_2,...,i_m=1}^{n}a_{i i_2...i_m} d^{(1)}_{i_2}...d^{(1)}_{i_m}x_{i_2}...x_{i_m},\\
		=&\sum_{ i_2,...,i_m=1}^{n}a_{ii_2...i_m}u_{i_2}...u_{i_m}\\
		=&(\A \u^{m-1})_i.
	\end{aligned}
\end{equation*}
Similarly, $\B D_2\x^{m-1}=-\B \v^{m-1}$, where $\u=D_1\x,\v=-D_2\x$. By rearranging terms, we get \begin{equation}	\begin{aligned}
		{\u}* {\v}&\leq{\bf 0},\\
		\A \u^{m-1}-\B \v^{m-1}&={\bf 0}.\\
	\end{aligned}
\end{equation}
Therefore by the ${\bf P}$ tensor pair property, $(\u,\v)=({\bf 0},{\bf 0}).$ But it contradicts our assumption. Hence we have our claim.\\
(ii)$\implies$(i): Assume contrary, there exists  nonzero $\z=(\x,\y)\in\Rr\times\Rr$, such that \begin{equation}
	\begin{aligned}
		{\x}* {\y}&\leq{\bf 0},\\
		\A \x^{m-1}-\B \y^{m-1}&={\bf 0},\\
	\end{aligned}
\end{equation}
With the help of ${\x}* {\y}\leq{\bf 0}$, we construct a vector $\u\in\Rr$, given as
$$u_i=\begin{cases}
	1 & \text{if} ~x_i>0\\ 	-1 & \text{if} ~x_i<0\\ 	-1 & \text{if} ~x_i=0\text{~and~} y_i> 0\\
	1 & \text{if} ~x_i=0\text{~and~} y_i< 0\\
	0 & \text{if} ~x_i=0\text{~and~} y_i= 0\\
\end{cases}.$$ Also we consider the diagonal matrices $D_1,D_2$ as
$$(D_1){_{ii}}=\begin{cases}
	\lvert x_i\rvert & \text{if} ~x_i\neq 0\\
	0 &\text{if} ~x_i=0\text{~and~} y_i\neq 0\\
	1       & \text{if} ~x_i=0 \text{~and~} y_i=0\\
\end{cases},~~(D_2)_{ii}=\lvert y_i\rvert,$$  where $(D_j)_{ii}$ is $ii^{th}$ entry of matrix $D_j$ for $j=1,2.$ We can see easily, $\x=D_1\u,\y=D_2(-\u)$ with $\diag(D_1+D_2)>{\bf 0}.$ As $\z=(\x,\y)$ is nonzero, $\u$ must be nonzero.
Thus $$\A(D_1{\bf u})^{m-1}-\B(D_2(-{\bf u}))^{m-1}={\bf 0},$$ As $m$ is even, $$\A D_1{\bf u}^{m-1}+\B D_2{\bf u}^{m-1}={\bf 0},$$ where $\diag(D_1+D_2)>{ \bf 0}$ and $
\u$ is a nonzero  vector. So we have, ${\rm det}(\A D_1+\B D_2)={0}. $ This contradicts the fact that, for arbitrary $n^{th}$ order nonnegative diagonal matrices $D_1,D_2$ with diag($D_1+D_2)>{\bf 0}$, $${\rm det}(\A D_1+\B D_2)\neq {0}.$$ Hence we have our claim.
\end{proof}

The above theorem is not applicable in the case of odd order tensor pairs, as demonstrated by the following example. Furthermore, the existence of an odd order ${\bf P}$ tensor pair is shown in the following example, despite the fact that there is no odd order $P$ tensor in TCP, see in \cite{you} .
\begin{example}\rm
Let $\A,\B\in T(3,2),$ with $a_{111}=a_{122}=1,b_{111}=b_{122}=-1$ and other $a_{i_1i_2i_3}=b_{i_1i_2i_3}=0.$ Let $\x,\y\in\Rr$ such that $\x*\y\leq {\bf 0},\A\xm-\B\ym={\bf 0}.$ By an easy calculation, $(\x,\y)=({\bf 0},{\bf 0}).$ Thus $\{\A,\B\}$ is a ${\bf P}$ tensor pair. Let $D_1,D_2$  be identity matrices of order $2$. By observation, $\diag(D_1+D_2)>{\bf 0}$ but ${\rm det}(\A D_1+\B D_2)={ 0}.$ Therefore above theorem does not hold for  odd order tensor.\end{example}
Next, we prove HTCP($\A,\B,{\bf q}$) has a nonempty compact solution set for every ${\bf q}\in\Rr$ if $\{\A,\B\}$  is a ${\bf P}$ tensor pair. 

To show this, first we prove the following result.
\begin{lemma}\label{nondeg}
Let ${\rm\Phi}(\x)=\A\x^{m-1}$. If \begin{itemize}
	\item [\rm(i)] $m$ is even.
	\item[\rm(ii)] ${\rm\Phi}(\x)={\bf 0} \iff \x={\bf 0}.$
\end{itemize} Then $\deg({\rm\Phi},{\bf 0})$ is nonzero.
\end{lemma}
\begin{proof}      
Let ${\rm\Phi}(\x)=\A\x^{m-1}$. Assume $D$ is a symmetric open bounded subset in $\Rr$ which contains ${\bf 0}$. As ${\rm\Phi}(\x)={\bf 0} \iff \x={\bf 0}$,  we see that ${\bf 0}\notin {\rm\Phi}(\partial D)$ and $$\deg({\rm\Phi},D,{\bf 0})=\deg({\rm\Phi},{\bf 0}).$$ Due to the fact $m$ is even, ${\rm\Phi}$ is an odd function. Therefore  by theorem \ref{odd}, $\deg({\rm\Phi},D,{\bf 0})=\deg({\rm\Phi},{\bf 0})\neq 0.$ 
\end{proof}

\begin{theorem}\label{exe}
Let $\A,\B\in E(m,n)$. If $\{\A,\B\}$ is a ${\bf P}$ tensor pair, {\rm HTCP($\A,\B,{\bf q}$)} has a nonempty compact solution set for every ${\bf q}\in\Rr.$
\end{theorem}
\begin{proof}       To prove this, we use degree theory tools and homotopy function.
Let us consider a homotopy, for $t\in[{0},1]$, $$F(\z,t)=t\begin{bmatrix}
	\y\\ \A \x^{m-1}
\end{bmatrix}+(1-t)\begin{bmatrix}
	\x\wedge \y\\ \A \x^{m-1}-\B\y^{m-1}
\end{bmatrix},$$where $\z=(\x,\y)\in\Rr\times\Rr$. At $t=0$ and $t=1$, $$F(\z,{0})={\rm\Psi}(\z)~ \text{and}~F(\z,1)=\begin{bmatrix}
	\y\\ \A \x^{m-1}
\end{bmatrix}.$$ To prove $ \deg(\A,\B)$ is nonzero, first we show the set, $Z=\{\z:F(\z,t)={\bf 0},~\text{where}~t\in[{0},1]\}$, is bounded. To prove our claim, we discuss the following cases here.\\ {\it Case: 1.} When $t={0}, $ we get ${\rm\Psi}(\z)={\bf 0}\implies\z={\bf 0}$ as $\{\A,\B\}$ is a ${\bf P}$ tensor pair. For $t=1, F(\z,1) ={\bf 0} \iff \z={\bf 0}$.\\ {\it Case: 2}. Now for $t\in({0},1)$, $$F(\z,t)={\bf 0},$$ $$\implies\begin{bmatrix}
	\x\wedge \y\\ \A \x^{m-1}-\B\y^{m-1}
\end{bmatrix}=-\alpha\begin{bmatrix}
	\y\\ \A \x^{m-1}
\end{bmatrix},$$ where $\alpha=\dfrac{t}{1-t}>{ 0},$ as $t\in({0},1).$	From the first row of above equation and proposition \ref{star}, we have $$\x\wedge  \y=-\alpha {\y}\implies \text{min}\{\x+\alpha {\y} ,(1+\alpha){\y}\}={\bf 0}.$$
From this, we get $\y \geq {\bf 0}$ and $({\x}+\alpha {\y}) * (1+\alpha){\y} ={\bf 0} $ which implies that ${\x} * {\y} \leq {\bf 0}.$ From the second row, \begin{equation*}
	\begin{aligned}
		\A \x^{m-1}-\B\y^{m-1}&=-\alpha \A \x^{m-1},\\ (1+\alpha)\A \x^{m-1}-\B\y^{m-1}&={\bf 0},\\
		\A (\alpha'\x)^{m-1}-\B\y^{m-1}&={\bf 0},	
	\end{aligned}
\end{equation*}
where $\alpha'=(1+\alpha)^{\frac{1}{m-1}}.$ Here $m$ is even so from the above equations, $$(\alpha'\x)*\y\leq {\bf 0},~	\A (\alpha'\x)^{m-1}-\B\y^{m-1}={\bf 0}.$$ As  $\{\A,\B\}$ is a ${\bf P}$ tensor pair, $((\alpha'\x),\y)=({\bf 0},{\bf 0})\implies \z=(\x,\y)=({\bf 0},{\bf 0}).$\\
Hence, $Z$ contains only the zero vector. By the homotopy invariance property of degree (D2), $$\deg(\A,\B)=\deg\big(F(\cdot,{ 0}),\Delta,{\bf 0}\big)=\deg\big(F(\cdot,1),\Delta,{\bf 0}\big).$$ Let ${\rm\Phi}(\x)=\A\xm$. Now for two arbitrary open bounded sets $\Delta_1$ and $\Delta_2$ containing ${\bf 0}$ in $\Rr$ such that $\Delta=\Delta_1\times\Delta_2$, by the Cartesian product (D3) and  (D1), we have, $$\deg(\A,\B)=\deg\big(F(\cdot,1),\Delta,{\bf 0}\big)=\deg\big({I},\Delta_1,{\bf 0}\big)\deg\big({\rm\Phi},\Delta_2,{\bf 0}\big).$$ As $\{\A,\B\}$ is ${\bf P} $ tensor pair, $\A\xm={\bf 0}\iff \x={\bf 0}.$ Thus  by lemma \ref{nondeg}, $$\deg(\A,\B)=\deg\big({\rm\Phi},\Delta_2,{\bf 0})\neq 0.$$ As ${\bf P}$ tensor pair is an ${\bf R}_0$ tensor pair,  $\{\A,\B\} $ is an ${\bf R}_0$ tensor pair with $\deg(\A,\B)\neq 0.$ Due to theorem \ref{Rrr}, $\S$ is nonempty compact  for every $\q$ in $\Rr$.
\end{proof}

The obvious follow-up question is \textquotedblleft what happens if $m$ is odd?" Does  ${\bf P}$ tensor pair  give  the existence of the HTCP's solution? We provide this answer negatively by the following example.
\begin{example}\rm
Let $\A,\B\in T(3,2).$ For $\{\A,\B\}=\{\I,-\I\}$ and $\x,\y\in\Rr$ such that 
\begin{equation*}
	\begin{aligned}
		\x*\y\leq &{\bf 0}\\ x^2_1+y^2_1=&{0},\\ x^2_2+y^2_2=&{ 0}.\\
	\end{aligned}
\end{equation*}
By above equations, $(\x,\y)=({\bf 0},{\bf 0}).$ Thus $\{\A,\B\}$ is a ${\bf P}$ tensor pair. Now let us take ${\q}=(0,-1)^t.$ As we we can see easily there does not exist any $(\x,\y)$ such that $$\x\wedge\y={\bf 0},~\A\xm-\B\ym=\q.$$ Therefore, in the  odd order  case ${\bf P}$ tensor pair does not  give existence of solution.
\end{example}
In HLCP, the $P$ pair ensures the uniqueness of the solution. However, one may wonder what happens in the HTCP case. Through simple observation, we can see that in even order cases, $\B$ is a $P$ tensor if and only if $\{\mathcal{I},\B\}$ is a ${\bf P}$ tensor pair. It is well known that TCP may not provide a unique solution when the given tensor is a $P$ tensor, as demonstrated in Example 3.2 in \cite{GUST}. Therefore, we can conclude that HTCP$(\I,\B,\q)$ may not have a unique solution when $\{\mathcal{I},\B\}$ is a ${\bf P}$ tensor pair. Thus, the ${\bf P}$ tensor pair does not necessarily ensure the uniqueness of the solution to the corresponding HTCP.
\subsection{Strong P tensor pair}
In TCP, global uniqueness solvable property is given by strong $P$ tensor see in \cite{GUST}. Motivated by the strong $P$ tensor definition, we define a strong ${\bf P}$ tensor pair for HTCP.
\begin{definition}\rm
	We say $\{\A,\B\}$ is strong ${\bf P}$ tensor pair if for $\x_i,\y_i \in\Rr,i\in\{1,2\}$,\begin{equation*}\begin{aligned}
			[(\x_1-\x_2)*(\y_1-\y_2)\leq {\bf 0},~(\A\x_1^{m-1}-\A\x_2^{m-1})-(\B\y_1^{m-1}-\B\y_2^{m-1})={\bf 0}\\  ~~\implies\x_1=\x_2,\y_1=\y_2].
		\end{aligned}
	\end{equation*} 
\end{definition}
\begin{prop}\label{stp}
	Let $\{\A,\B\}$ be a strong ${\bf P}$ tensor pair. Then the followings hold:
	\begin{itemize}
		\item [\rm(i)] $\A$ and $\B$ are even order tensors.
		\item[\rm(ii)] $F(\x)=\A\xm$ and $G(\y)=\B\ym$ are injective functions.
		\item [\rm(iii)] $\deg(\A,\B)$ is non-zero.
	\end{itemize}
\end{prop}
\begin{proof}      (i):
Let $\x_1=\x,\x_2=-\x$ and $\y_i={\bf 0},$ for $i=1,2$. Assume contrary $m$ is odd so $\A\x^{m-1}=\A(-{\bf x})^{m-1}$. Then $$(\x_1-\x_2)*(\y_1-\y_2)= {\bf 0},~(\A\x_1^{m-1}-\A\x_2^{m-1})-(\B\y_1^{m-1}-\B\y_2^{m-1})={\bf 0}.$$ Therefore $\x_1=\x_2.$ But it contradicts our assumption i.e., $\x_1=\x,\x_2=-\x$. Hence $m$ is even. \\
(ii): Let $F(\x_1)=F(\x_2)$ and $\y_1={\bf 0}=\y_2$ then we get $$(\x_1-\x_2)*(\y_1-\y_2)= {\bf 0},~(\A\x_1^{m-1}-\A\x_2^{m-1})-(\B\y_1^{m-1}-\B\y_2^{m-1})={\bf 0}.$$ As $\{\A,\B\}$ be a strong ${\bf P}$ tensor pair, $\x_1=\x_2.$ Hence $F$ is an injective function. Similarly, $G$ is also an injective function.\\
(iii): As  $\{\A,\B\}$ is a strong ${\bf P}$ tensor pair,  $\{\A,\B\}$ is a ${\bf P}$ tensor pair. Therefore from Theorem \ref{exe},  $\deg(\A,\B)\neq 0.$
\end{proof}

Next, we prove an equivalence relation for strong $P$ tensor and strong ${\bf P}$ tensor pair.
\begin{prop} 
	Let $\B\in T(m,n)$. Then $\B$ is a strong $ P$ tensor if and only if $\{\mathcal{I},\B\}$ is a strong ${\bf P}$ tensor pair.
	
\end{prop}
\begin{proof}      
Let $\x_i,\y_i \in\Rr,i\in\{1,2\}$ such that $$(\x_1-\x_2)*(\y_1-\y_2)\leq {\bf 0},~({\mathcal{I}}\x_1^{m-1}-\mathcal{I}\x_2^{m-1})-(\B\y_1^{m-1}-\B\y_2^{m-1})={\bf 0}.$$ As we know $m$ is even, $(\x_1-\x_2)*(\y_1-\y_2)\leq {\bf 0}\implies (\x^{[m-1]}_1-\x^{[m-1]}_2)*(\y_1-\y_2)\leq {\bf 0}.$ Also, $\x_1^{[m-1]}-\x_2^{[m-1]}=(\B\y_1^{m-1}-\B\y_2^{m-1})$. Thus $(\y_1-\y_2)*(\B\y_1^{m-1}-\B\y_2^{m-1})\leq{\bf 0}$. As $\B$ is strong $P$ tensor, $\y_1=\y_2$. This gives $\x_1=\x_2.$ Hence we have our claim. Converse follows similarly.
\end{proof}

\begin{theorem}
	Let $\{\A,\B\}$ be a strong ${\bf P}$ tensor pair then  {\rm HTCP$(\A,\B,{\bf q})$} has the unique solution for any $\q\in\Rr$.
\end{theorem}
\begin{proof}       The existence of solution comes from item (iii) of proposition \ref{stp}. For uniqueness, let there exist  $\x_i,\y_i \in\Rr,i\in\{1,2\}$ and $\q\in\Rr$ such that 	\begin{equation}\begin{aligned}
		{\x_1}\wedge { \y_1}&={\bf 0},\A\xm_1-\B\ym_1={\bf q}.\\
		{\x_2}\wedge { \y_2}&={\bf 0},\A\xm_2-\B\ym_2={\bf q.}\end{aligned}\end{equation} Above system, we can write as 
\begin{equation}\begin{aligned} (\x_1-\x_2)*(\y_1-\y_2)\leq &{\bf 0},\\ (\A\x_1^{m-1}-\A\x_2^{m-1})-(\B\y_1^{m-1}-\B\y_2^{m-1})=&{\bf 0}.\end{aligned}
\end{equation}
As $\{\A,\B\}$ are strong ${\bf P}$ tensor pair, $\x_1=\x_2$ and $\y_1=\y_2.$ Hence proved.\end{proof}
\section{Conclusion}
 We introduced the HTCP in this paper, followed by a study of its boundedness and the existence of the solution. In order to prove the boundedness and existence of the solution to the HTCP, we introduced the concepts of ${\bf R_0}$,${\bf R}$ and ${\bf P}$ tensor pair. We offer an equivalent characterization of the ${\bf P}$ tensor pair with  one of the main existence result. Lastly, we have given a sufficient condition for the uniqueness of the solution for HTCP with the help of a strong ${\bf  P}$ tensor pair.
\section*{Declaration of Competing Interest} The authors have no competing interests.
\section*{Acknowledgements} The first author is a CSIR-SRF fellow. He wants to thank the Council of Science \& Industrial Research(CSIR) for the financial support.

\end{document}